\documentclass[12pt]{amsart}

\usepackage{fullpage}
\usepackage{amsfonts}
\usepackage{enumerate}
\usepackage{xcolor}
\usepackage{verbatim}

\newcommand{\R}{\mathbb{R}}
\newcommand{\C}{\mathbb{C}}
\newcommand{\N}{\mathbb{N}}

\newcommand{\Q}{\mathbb{Q}}

\newcommand{\ve}{\varepsilon}
\newcommand{\lan}{\langle}
\newcommand{\ran}{\rangle}

\numberwithin{equation}{section}
\theoremstyle{plain} 

\newtheorem{theorem}{Theorem}[section]
\newtheorem{corollary}[theorem]{Corollary}
\newtheorem{lemma}[theorem]{Lemma}
\newtheorem{proposition}[theorem]{Proposition}

\theoremstyle{definition}
\newtheorem{definition}{Definition}[section]

\theoremstyle{remark}
\newtheorem{remark}{Remark}[section]
\newtheorem{ex}{Example}[section]

\begin{document}

\title{Lyapunov's Theorem for continuous frames}

\author{Marcin Bownik}

\address{Department of Mathematics, University of Oregon, Eugene, OR 97403--1222, USA}
\address{
Institute of Mathematics, Polish Academy of Sciences, ul. Wita Stwosza 57,
80--952 Gda\'nsk, Poland}
\email{mbownik@uoregon.edu}

\date{\today}



\thanks{The author was partially supported by NSF grant  DMS-1665056 and by a grant from the Simons Foundation \#426295. }

\begin{abstract} 
Akemann and Weaver (2014) have shown a remarkable extension of Weaver's $KS_r$ Conjecture (2004) in the form of approximate Lyapunov's theorem. This was made possible thanks to the breakthrough solution of the Kadison-Singer problem by Marcus, Spielman, and Srivastava (2015). In this paper we show a similar type of Lyapunov's theorem for continuous frames on non-atomic measure spaces. In contrast with discrete frames, the proof of this result does not rely on the recent solution of the Kadison-Singer problem.
\end{abstract}

\maketitle

\section{Introduction}

The classical Lyapunov's theorem states that the range of a non-atomic vector-valued measure with values in $\R^n$ is a convex and compact subset of $\R^n$. In contrast, the range of a vector measure with values in an infinite dimensional Banach spaces might not be convex. This leads to the problem of identifying vector-valued measures that have convex range. Some early results on this topic can be found in the monograph of Diestel and Uhl \cite[Chap. IX]{DU}. For example, Uhl's theorem \cite{Uh} gives sufficient conditions for the convexity of the closure of the range of a non-atomic vector-valued measure.

Kadets and Schechtman \cite{KS2} introduced the Lyapunov property of a Banach space as follows: the closure of a range of every non-atomic vector measure is convex. They have shown that $c_0$ space and $\ell^p$ spaces for $1\le p <\infty$, $p\ne 2$, satisfy the Lyapunov property. However, it is known that $\ell^2$ fails this property. The counterexample is $L^2([0,1])$-valued measure that assigns to any measurable $E \subset [0,1]$ a characteristic function $\chi_E$.

The other interest in Lyapunov's theorem comes from operator algebras in the work of Akemann and Anderson \cite{AA} who investigated the connection with the long-standing Kadison-Singer problem \cite{KS}. The breakthrough solution of the Kadison-Singer problem by Marcus, Spielman, and Srivastava \cite{MSS} has had a great impact on the area. A remarkable result of Akemann and Weaver \cite{AW} is an interesting generalization of newly confirmed Weaver's $KS_r$ Conjecture \cite{We} in the form of approximate Lyapunov's theorem. Their result states that the set of all partial frame operators corresponding to a given frame (or more generally a Bessel sequence) in a Hilbert space $\mathcal H$ forms an approximately convex subset of $\mathcal B(\mathcal H)$. The degree of approximation is dependent on how small the norms of frame vectors are. The exact formulation can be found in Section \ref{S3}.

In this paper we study a related problem for continuous frames defined on non-atomic measure spaces.  A concept of continuous frame, which is a generalization of the usual (discrete) frame, was proposed independently by Ali, Antoine, and Gazeau \cite{aag} and by G. Kaiser \cite{Ka}, see also \cite{aag2, FR, GH}.

\begin{definition}\label{cf} Let $\mathcal H$ be a separable Hilbert spaces and let $(X,\mu)$ be a measure space. A family of vectors $\{\phi_t\}_{t\in X}$ is a {\it continuous frame} over $X$ for $\mathcal H$ if:
\begin{enumerate}[(i)]
\item for each $f\in \mathcal H$, the function $X \ni t \mapsto \langle f , \phi_t \rangle \in \C$ is measurable, and
\item there are constants $0<A \le B< \infty$, called {\it frame bounds}, such that 
\begin{equation}\label{cf1}
A||f||^2 \le \int_X |\langle f, \phi_t \rangle|^2 d\mu (t) \le B ||f||^2 \qquad\text{for all }f\in\mathcal H.
\end{equation}
\end{enumerate}
When $A=B$, the frame is called {\it tight}, and when $A=B=1$, it is a {\it continuous Parseval frame}. More generally, if only the upper bound holds in \eqref{cf1}, that is $A=0$, we say that $\{\phi_t\}_{t\in X}$ is a {\it continuous Bessel family} with bound $B$.
\end{definition}

Every continuous frame defines a positive operator-valued measure (POVM) on $X$, see \cite{MHC}. To any measurable subset $E\subset X$, we assign a partial frame operator $S_{\phi,E}$ given by
\[
S_{\phi,E} f  = \int_E \lan f, \phi_t \ran \phi_t d\mu(t) \qquad\text{for } f\in\mathcal H.
\]
The main result of this paper, Theorem \ref{lyu}, shows that the closure of the range of such POVM is convex if $\mu$ is non-atomic. This result should be contrasted with the special case of POVM known as spectral measure or projection-valued measure (PVM). Such measures appear in the formulation of the spectral theorem for self-adjoint, or more generally, normal operators. The range of PVM is far from being convex since it consists solely of projections. In particular it contains zero $\mathbf 0$ and identity $\mathbf I$ operators, but not $\frac{1}{2}\mathbf I$. This naturally leads  to the problem of classifying those POVMs for which the closure of the range is convex. In Section \ref{S4} we show an extension of our main theorem to POVMs given by measurable positive compact operator-valued mappings.

Unlike the approximate Lyapunov's theorem for discrete frames by Akemann and Weaver \cite{AW}, its counterpart for continuous frames does not rely on the solution of the Kadison-Singer problem. This might initially look surprising, but it is consistent with the past experience. Indeed, Kadison and Singer \cite{KS} have shown that pure states on continuous MASA (maximal abelian self-adjoint algebra) in general have non-unique extensions to the entire algebra $\mathcal B(\mathcal H)$. In fact, the same is true for MASA with non-trivial continuous component. In contrast, the same problem for  discrete MASA has been a very challenging topic of research with a large number of equivalent formulations, see \cite{Bow, CT}. Finally, it is worth mentioning another recent result about continuous frames which actually relies on the solution of the Kadison-Singer problem. Freeman and Speegle \cite{FS} have solved the discretization problem posed by Ali, Antoine, and Gazeau \cite{aag2}. This problem asks which continuous frames can be sampled to yield a discrete frame.

\section{Measure theoretic reductions}

We start by making some remarks about measurability condition in Definition \ref{cf}.

\begin{remark}\label{rcf}
Since $\mathcal H$ is separable, by the Pettis Measurability Theorem \cite[Theorem II.2]{DU}, the weak measurability (i) is equivalent to (Bochner) strong measurability on $\sigma$-finite measure spaces $X$. That is, $t \mapsto \phi_t$ is a pointwise a.e. limit of simple measurable functions. Moreover, by \cite[Corollary II.3]{DU}, every measurable function $\phi: X\to \mathcal H$ is a.e. uniform limit of a sequence of countably-valued measurable functions. Although these results were stated in \cite{DU} for finite measure spaces, they also hold for $\sigma$-finite measure spaces.
\end{remark}

Since we work only with separable Hilbert spaces, we can safely assume that the measure space $(X,\mu)$ is $\sigma$-finite. Indeed, by Proposition \ref{p1} every continuous frame, or more generally a continuous Bessel family, is supported on a $\sigma$-finite set.

\begin{proposition}\label{p1} Suppose that $\{\phi_t\}_{t\in X}$ is a continuous Bessel family, then its support
$
\{t\in X: \phi_t \ne 0\}
$
is a $\sigma$-finite subset of $X$.
\end{proposition}

\begin{proof}
Let $\{e_i\}_{i\in I}$ be an orthonormal basis of $\mathcal H$, where the index set $I$ is at most countable. For any $n\in \N$ and $i\in I$, by Chebyshev's inequality \eqref{cf1} yields
\[
\mu(\{t\in X: |\langle e_i, \phi_t \rangle|^2> 1/n \}) \le Bn<\infty.
\]
Hence, the set 
\[
\{t\in X: \phi_t \ne 0\}= \bigcup_{i\in I} \bigcup_{n\in\N} \{t\in X: |\langle e_i, \phi_t \rangle|^2> 1/n \} 
\]
is a countable union of sets of finite measure.
\end{proof}

It is convenient to define a concept of a weighted frame operator as follows. This is a special case of a continuous frame multiplier introduced by Balazs, Bayer, and Rahimi \cite{BBR}; for a discrete analogue, see \cite{Ba}.

\begin{definition}
Suppose that $\{\phi_t\}_{t\in X}$ is a continuous Bessel family.
For any measurable function $\tau: X \to [0,1]$, 
define a {\it weighted frame operator}
\[
S_{\sqrt{\tau}\phi,X} f= \int_X \tau(t) \langle f, \phi_t \rangle \phi_t d\mu(t)
\qquad f\in \mathcal H.
\]
\end{definition}

\begin{remark} A quick calculation shows that $\{\sqrt{\tau(t)} \phi_t\}_{t\in X}$ is also a continuous Bessel family with the same bound as $\{\phi_t\}_{t\in X}$. Hence, a weighted frame operator is merely the usual frame operator associated to $\{\sqrt{\tau(t)} \phi_t\}_{t\in X}$.
\end{remark}

Using Proposition \ref{p1} we will deduce the following approximation result for continuous frames.

\begin{lemma}\label{approx}
Let $(X,\mu)$ be a measure space and let $\mathcal H$ be a separable Hilbert space.
Suppose that $\{\phi_t\}_{t\in X}$ is a continuous Bessel family in $\mathcal H$. Then for every $\ve>0$, there exists a continuous Bessel family $\{\psi_t\}_{t\in X}$, which takes only countably many values, such that for any measurable function $\tau: X \to [0,1]$ we have
\[
||S_{\sqrt{\tau}\phi,X} - S_{\sqrt{\tau}\psi,X}||<\ve.
\]
\end{lemma}

\begin{proof}
By Proposition \ref{p1} we can assume that $(X,\mu)$ is $\sigma$-finite. Since a measurable mapping is constant a.e. on atoms, and there are at most countably many atoms, we can assume that $\mu$ is a non-atomic measure.
Define the sets $X_0=\{t\in X: ||\phi_t||<1\}$ and
\[
X_n= \{t\in X: 2^{n-1} \le ||\phi_t||< 2^n \}, \qquad n\ge 1.
\]
Then, for any $\ve>0$, we can find a partition $\{X_{n,m}\}_{m\in \N}$ of each $X_n$ such that $\mu(X_{n,m}) \le 1$ for all $m\in\N$. Then, we can find a countably-valued measurable function $\{\psi_t\}_{t\in X}$ such that 
\[
||\psi_t - \phi_t || \le \frac{\ve}{4^n 2^m} \qquad\text{for }t\in X_{n,m}.
\]
Take any $f\in\mathcal H$ with $||f||=1$.
Then, for any $t\in X_{n,m}$,
\[
\begin{aligned}
||\langle f, \psi_t \rangle|^2 - |\langle f, \phi_t \rangle|^2|
& = |\langle f, \psi_t-\phi_t \rangle||\langle f, \psi_t +\phi_t \rangle| \le ||\psi_t - \phi_t||(||\psi_t||+||\phi_t||)
\\
&\le \frac{\ve}{4^n 2^m}(2^n+\ve + 2^n) \le \frac{3\ve}{2^n2^m} .
\end{aligned}
\]
Integrating over $X_{n,m}$ and summing over $n\in\N_0$ and $m\in\N$ yields
\[
\int_X ||\langle f, \psi_t \rangle|^2-|\langle f, \phi_t \rangle|^2 | d\mu(t)  \le  \sum_{n=0}^\infty\sum_{m=1}^\infty \frac{3\ve}{2^n2^m} \mu(X_{n,m}) \le 6\ve.
\]
Using the fact that $S_{\sqrt{\tau}\phi,X}$ is self-adjoint, we have
\[
\begin{aligned}
||S_{\sqrt{\tau}\phi,X} - S_{\sqrt{\tau}\psi,X}||
& = \sup_{||f||=1} |\langle (S_{\sqrt{\tau}\phi,X} - S_{\sqrt{\tau}\psi,X})f,f \rangle |
\\
&= \sup_{||f||=1} \bigg| \int_X \tau(t) (|\langle f, \psi_t \rangle|^2-|\langle f, \phi_t \rangle|^2) d\mu(t)  \bigg| \le 6\ve.
\end{aligned}
\]
Since $\ve>0$ is arbitrary, this completes the proof.
\end{proof}

\begin{remark}\label{r2.3}
Suppose $\{\psi_t\}_{t\in X}$ is a continuous frame which takes only countably many values as in Lemma \ref{approx}. Then for practical purposes, such a frame can be treated as a discrete frame. Indeed, there exists a sequence $\{\tilde \psi_n\}_{n\in \N}$ in $\mathcal H$ and a partition $\{X_n\}_{n\in\N}$ of $X$ such that
\begin{equation}\label{rem2}
\psi_t= \tilde \psi_n \quad\text{for all }t\in X_n, \ n\in \N.
\end{equation}
Since $\{\psi_t\}_{t\in X}$ is Bessel, we have $\mu(X_n)<\infty$ for all $n$ such that $\tilde \psi_n \ne 0$. Define vectors
\[
\tilde \psi_n =\sqrt{\mu(X_n)}\psi_n \qquad  n\in\N.
\]
Then, for all $f\in \mathcal H$,
\begin{equation}\label{rem3}
\int_X |\langle f, \psi_t \rangle|^2 d\mu (t) = \sum_{n\in \N} \int_{X_n} |\langle f, \psi_t \rangle|^2 d\mu (t) = \sum_{n\in \N} |\langle f, \tilde \psi_n \rangle|^2.
\end{equation}
Hence, $\{ \tilde \psi_n \}_{n\in\N}$ is a discrete frame and its frame operator coincides with that of a continuous frame $\{\psi_t\}_{t\in X}$. 
\end{remark}

In particular, if the measure space $X$ is $\sigma$-finite and atomic, then any continuous frame on $X$ takes only countably many values. That is, $X$ has a partition into atoms $\{X_n\}_{n\in\N}$. Then, the procedure in Remark \ref{r2.3} boils down to rescaling of atoms, which identifies atomic measure space $X$ with the counting measure on $\N$. Since every measure space decomposes into atomic and non-atomic components, we would like to investigate in detail continuous frames on non-atomic measure spaces $X$. As we will see below, such frames can be reduced to the case of Lebesgue measure on a subinterval of $\R$.

Our first reduction result shows that without loss of generality we can assume that the measure algebra associated with $(X,\mu)$ is separable. Let $\mathcal M$ denote the $\sigma$-algebra of $(X,\mu)$. Recall \cite[Sec. 40]{Ha} that a {\it measure algebra} associated with measure space $(X,\mu)$ consists of equivalence classes of measurable sets under the relation 
\[
E , F \in \mathcal M \qquad E \sim F \iff \mu(E \Delta F)=0,
\]
where $\Delta$ is a symmetric difference. Then, the set of measurable sets of finite measure becomes a metric space with the distance
\[
\rho(E,F) = \mu(E \Delta F) \qquad E,F \in\mathcal M.
\]
A measure algebra associated with $(X,\mathcal M, \mu)$ is {\it separable} if the corresponding metric space is separable. Then, we have the following fact.

\begin{proposition}\label{p2} Suppose that $\{\phi_t\}_{t\in X}$ is a continuous Bessel family defined on a $\sigma$-finite measure space $(X,\mathcal M,\mu)$. 
Let $\mathcal M' \subset \mathcal M$ be a $\sigma$-algebra generated by the sets
\[
\{ t\in X: \phi_t \in U \}, \qquad\text{where $U \subset \mathcal H$ is open}.
\]
Then, a measure algebra associated with $(X,\mathcal M', \mu)$ is  separable. 
\end{proposition}

\begin{proof}
Let $\mathcal D$ be a countable dense subset of $\mathcal H$. Then, $\sigma$-algebra $\mathcal M'$ is generated by the sets of the form
\[
\{ t\in X: ||f-\phi_t||<q \},
\qquad\text{where }f\in\mathcal D, \ 0<q \in \Q.
\]
Since balls in $\mathcal H$, and hence open sets in $\mathcal H$, are Borel sets with respect to the weak topology on $\mathcal H$, the above sets belong to $\mathcal M$. Consequently, $\sigma$-algebra $\mathcal M'$ is countably generated. By \cite[Theorem B in \S40]{Ha}, the metric space of $\mathcal M'$-measurable sets is separable.
\end{proof}

Combining Propositions \ref{p1} and \ref{p2} we obtain the following result. Corollary \ref{p3} shows that a continuous frame over any measure space can be reduced to a continuous frame over a separable measure algebra.

\begin{corollary}\label{p3}
Let $\mathcal H$ be a separable Hilbert spaces and let $(X,\mathcal M, \mu)$ be a measure space. Suppose that $\{\phi_t\}_{t\in X}$ is a continuous Bessel family over $X$ in $\mathcal H$. Then there exists $\sigma$-finite subset $X' \subset X$ and a $\sigma$-algebra 
$
\mathcal M' \subset \{ E \cap X': E\in\mathcal M\}
$
such that:
\begin{enumerate}[(i)]
\item $\phi_t =0$ for all $t\in X \setminus X'$,
\item the restriction $\{\phi_t\}_{t\in X'}$ is a continuous Bessel family over $(X',\mathcal M',\mu)$, and
\item the measure algebra of $(X',\mathcal M',\mu)$ is separable.
\end{enumerate}
\end{corollary}

We will use the classical isomorphism theorem for measure algebras due to Carath\'eodory, see \cite[Theorem 9.3.4]{Bo} or \cite[Theorem C in \S41]{Ha}.

\begin{theorem}[Carath\'eodory]\label{p4} 
Every separable, non-atomic,  measure algebra of a probability space is isomorphic to the measure algebra of the Lebesgue unit interval.
\end{theorem}

As a consequence of Theorem \ref{p4} we have:

\begin{proposition}\label{p5} Suppose that $(X,\mu)$ is a non-atomic, $\sigma$-finite measure space such that its measure algebra is separable. Let $\phi: X \to \mathcal H$ be a weakly measurable function. Then there exists a weakly measurable function $\psi: [0,\mu(X)) \to \mathcal H$, which has the same distribution as $\phi$. That is, 
\begin{equation}\label{dst}
\mu (\phi^{-1}(U)) = \lambda( \psi^{-1}(U)) \qquad\text{ for any open $U \subset \mathcal H$},
\end{equation}
where $\lambda$ denotes the Lebesgue measure on $\mathcal \R$.
\end{proposition}

\begin{proof}
If $\mu(X)=\infty$, then there exists a sequence of disjoint measurable subsets $\{X_m\}_{m\in \N}$ of $X$ such that
\[
X = \bigcup_{m=1}^\infty X_m \qquad\text{and}\qquad
\mu(X_m)=1 \quad\text{for all }m\in\N.
\]
By Theorem \ref{p4}, the measure algebra of each $(X_m,\mu|_{X_m})$ is isomorphic with $([m-1,m],\lambda)$, where $\lambda$ denotes the Lebesgue measure. These isomorphisms induce a global isomorphism of a measure algebra of $(X,\mu)$ with that $([0,\infty),\lambda)$, see \cite[\S 41, Ex. 6]{Ha}. If $\mu(X)<\infty$, the measure algebra $(X,\mu)$ is isomorphic with that of $([0,\mu(X)),\lambda)$ by a simple rescaling of Theorem \ref{p4}. 

Now, let $\phi: X\to \mathcal H$ be weakly measurable. If $\phi$ takes at most countably many values, then the isomorphisms of measure algebras yields $\psi: [0,\mu(X)) \to \mathcal H$, which has the same distribution as $\phi$. In general, by Remark \ref{p1} $\phi$ is an a.e. uniform limit of a sequence of measurable functions $\phi_n:  X \to \mathcal H$, $n\in \N$, which take at most countably many values. The isomorphism of measure algebras yields $\psi_n:  [0,\mu(X)) \to \mathcal H$, $n\in\N$, such that:
\begin{itemize}
\item[(i)]
$\psi_n$ has the same distribution as $\phi_n$ for every $n\in \N$,
\item[(ii)]
$\psi_n-\psi_m$ has the same distribution as $\phi_n-\phi_m$ for every $n,m \in \N$.
\end{itemize}
By (ii), the sequence $\{\psi_n\}_{n\in\N}$ converges a.e. uniformly to some limiting function $\psi$. 
In particular, functions $\psi_n$ converge in measure to $\psi$ as $n\to \infty$ if $\mu(X)<\infty$. If $\mu(X)=\infty$, then restrictions $\psi_n|_{X_m}$ converge in measure to $\psi|_{X_m}$ for each $m\in \N$. In either case, (i) implies that $\phi$ and $\psi$ have the same distribution.
\end{proof}

Combining Corollary \ref{p3} and Proposition 2.6 yields the following result. Theorem \ref{p6} shows that from measure theoretic viewpoint a continuous frame on non-atomic measure space can be reduced to the setting of Lebesgue measure on an interval.

\begin{theorem}\label{p6}
Let $\mathcal H$ be a separable Hilbert spaces and let $(X,\mathcal M, \mu)$ be a non-atomic measure space. Suppose that $\{\phi_t\}_{t\in X}$ is a continuous Bessel family over $X$ in $\mathcal H$. Then there exists a continuous Bessel family $\{\psi_t\}_{t\in I}$ over interval $I=[0,\mu(X))$, which has the same distribution as $\{\phi_t\}_{t\in X}$ on its support, i.e., \eqref{dst} holds for any open $U \subset \mathcal H \setminus \{0\}$.
\end{theorem}

\begin{proof}
If we restrict $\{\phi_t\}_{t\in X'}$ to its support $X'=\{t\in X: \phi_t \ne 0\}$, Corollary \ref{p3} shows that we have a continuous Bessel family over $\sigma$-finite and separable measure algebra on $X'$. Since the underlying measure space is non-atomic, Proposition \ref{p5} yields a continuous Bessel family $\{\phi_t\}_{t\in I'}$, where $I'=[0,\mu(X'))$, which has the same distribution as $\{\psi_t\}_{t\in X'}$. If $\mu(X')<\mu(X)$, then setting $\psi_t=0$ for $t\in [\mu(X'),\mu(X))$ yields the required 
continuous Bessel family over $I=[0,\mu(X))$. It has the same distribution as $\{\psi_t\}_{t\in X}$ neglecting the set on which it vanishes.
\end{proof}

\section{Lyapunov's theorem}\label{S3}

Akemann and Weaver \cite{AW} have shown an interesting generalization of Weaver's $KS_r$ Conjecture \cite{We} in the form of approximate Lyapunov's theorem. This was made possible thanks to the breakthrough solution of the Kadison-Singer problem \cite{CT, KS} by Marcus, Spielman, and Srivastava \cite{MSS}. In this section we show a similar type of result for continuous frames. 

For $\phi \in \mathcal H$, let $\phi\otimes\phi$ denote a rank one operator given by
\[
(\phi\otimes\phi)(f) = \langle f, \phi \rangle \phi 
\qquad\text{for }f\in\mathcal H.
\]

The following lemma is an infinite dimensional formulation of a result due to Akemann and Weaver \cite[Lemma 2.3]{AW}. The proof of this fact heavily depends on a qualitative version of Weaver's $KS_r$ Conjecture shown by Marcus, Spielman, and Srivastava in \cite[Corollary 1.5]{MSS}.

\begin{lemma}[Akemann and Weaver]
\label{awl}
There exists a universal constant $C>0$ such that the following holds.
Suppose $\{\phi_i\}_{i\in I}$ is a Bessel family with bound $1$ in a separable Hilbert space $\mathcal{H}$, which consists of vectors of norms $\|\phi_i\|^2\leq \ve$, where $\ve>0$. Let $S$ be its frame operator. Then for any $0 \le \tau \le 1$, there exists a subset $I_0 \subset I$ such that
\[
\bigg\| \sum_{i\in I_0} \phi_i \otimes \phi_i - \tau S \bigg\| \le C \ve^{1/4}.
\]
\end{lemma}

\begin{proof}
Lemma \ref{awl} has been shown in great detail in finite dimensional case in \cite[Lemma 2.3]{AW}. As mentioned in \cite[Section 3]{AW}, it extends to the infinite dimensional case. For the sake of completeness, we merely indicate the strategy for proving it.

First, note that we can relax the Parseval frame assumption in \cite[Lemma 2.1]{AW} by the Bessel sequence condition with bound $1$. Then, using the pinball principle \cite[Theorem 6.9]{BCMS} we can generalize \cite[Lemma 2.1]{AW} to the infinite dimensional setting. Alternatively, we can use the fact that any sequence of partitions of the compact space $\{1, \ldots, r\}^\N$ has a cluster point, see \cite[Theorem 3.1]{AW}. The details are explained in the proof of \cite[Lemma 2.8]{Bow}, which shows how to deduce infinite dimensional Weaver's $KS_r$ conjecture from its finite dimensional counterpart. Hence, \cite[Corollary 2.2]{AW} also extends to the setting of a separable Hilbert space $\mathcal H$. Finally, the proof of \cite[Lemma 2.3]{AW} extends verbatim to infinite dimensions.  
\end{proof}

Lemma \ref{awl} implies approximate Lyapunov's theorem for discrete frames due to Akemann and Weaver \cite[Theorem 2.4]{AW}. This result also holds in the infinite dimensional setting, where $C>0$ denotes a universal constant.

\begin{theorem}[Akemann and Weaver]
\label{aw}
Suppose $\{\phi_i\}_{i\in I}$ is a Bessel family with bound $1$ in a separable Hilbert space $\mathcal{H}$, which consists of vectors of norms $\|\phi_i\|^2\leq \ve$, where $\ve>0$. Suppose that $0\le \tau_i \le 1$ for all $i\in I$. Then, there exists a subset of indices $I_0 \subset I$ such that
\begin{equation}\label{aw0}
\bigg\| \sum_{i\in I_0} \phi_i \otimes \phi_i - \sum_{i\in I} \tau_i \phi_i \otimes \phi_i \bigg\| \le C  \ve^{1/8}.
\end{equation}
\end{theorem}

Theorem \ref{aw} can be used to show Lyapunov's theorem for continuous frames over non-atomic measure spaces. However, Theorem \ref{lya} can also be shown directly without employing Theorem \ref{aw}, which relies on the solution of the Kadison-Singer problem. As in the discrete case of Theorem \ref{aw}, the lower frame bound does not play any role. Hence, all of our results hold for continuous Bessel families. 

\begin{theorem}\label{lya}
Let $(X,\mu)$ be a non-atomic $\sigma$-finite measure space.
Suppose that $\{\phi_t\}_{t\in X}$ is a continuous Bessel family in $\mathcal H$.  For any measurable function $\tau: X \to [0,1]$, 
consider a weighted frame operator
\[
S_{\sqrt{\tau}\phi,X} f= \int_X \tau(t) \langle f, \phi_t \rangle \phi_t d\mu(t)
\qquad f\in \mathcal H.
\]
Then, for any $\ve>0$, there exists a measurable set $E \subset X$ such that
\begin{equation}\label{lya1}
||S_{\phi,E} - S_{\sqrt{\tau}\phi,X}||<\ve.
\end{equation}
\end{theorem}

\begin{proof}
Let $\{\psi_t\}_{t\in X}$ be continuous Bessel family as in Lemma \ref{approx}. Since it takes only countably many values, there exists a sequence $\{\tilde \psi_n\}_{n\in N}$ in $\mathcal H$ and a partition $\{X_n\}_{n\in\N}$ of $X$ such that
\begin{equation}\label{lya2}
\psi_t= \tilde \psi_n \quad\text{for all }t\in X_n, \ n\in \N.
\end{equation}
Since $\{\psi_t\}_{t\in X}$ is Bessel, we have $\mu(X_n)<\infty$ for all $n$ such that $\tilde \psi_n \ne 0$.
Moreover, by subdividing sets $X_n$ if necessary we can assume that 
\begin{equation}\label{lya3}
||\tilde \psi_n||^2 \mu(X_n) \le \ve^2 \qquad\text{for all }n\in\N.
\end{equation}
This is possible since the measure $\mu$ is non-atomic.
Then, the continuous frame $\{\psi_t\}_{t\in X}$ is equivalent to a discrete frame 
\[
\{\phi_n =\sqrt{\mu(X_n)}\psi_n \}_{n\in\N}.
\]
More precisely, for any measurable function $\tau: X \to [0,1]$, the frame operator $S_{\sqrt{\tau}\psi,X}$ of a continuous Bessel family $\{\sqrt{\tau(t)}\psi_t\}_{t\in X}$ coincides with the frame operator of a discrete Bessel sequence
\begin{equation}\label{lya4}
\{\sqrt{\tau_n} \phi_n\}_{n\in \N} \qquad\text{where }
 \tau_n=\int_{X_n}\tau(t) d\mu(t).
\end{equation}
At this moment, one is tempted to apply Theorem \ref{aw}, since  \eqref{lya3} guarantees that its assumptions are satisfied. This might require rescaling to guarantee that the Bessel bound is $\le1$.
Hence, there exists an index set $I_0 \subset I:=\N$ such that \eqref{aw0} holds. By \eqref{lya2} and \eqref{lya4},
\[
\sum_{n\in I_0} \phi_n \otimes \phi_n = \int_E \psi_t \otimes \psi_t d\mu(t) = S_{\psi, E}
\qquad\text{where }
E= \bigcup_{n\in I_0} X_n.
\] 
Hence, by Lemma \ref{approx} and Theorem \ref{aw} we have
\[
\begin{aligned}
||S_{\phi,E} - S_{\sqrt{\tau}\phi,X}|| &\le
||S_{\phi,E} - S_{\psi,E}||+ ||S_{\psi,E}- S_{\sqrt{\tau}\psi,X}||+||S_{\sqrt{\tau}\psi,X} - S_{\sqrt{\tau}\phi,X}||
\\
& \le \ve + \bigg\| \sum_{n\in I_0} \phi_n \otimes \phi_n - \sum_{n\in \N} \tau_n \phi_n \otimes \phi_n \bigg\| +\ve
\le 2\ve+ C\ve^{1/8}.
\end{aligned}
\]

However, one can easily avoid using Theorem \ref{aw} as follows. Since $\mu$ is non-atomic, we can find subsets $E_n \subset X_n$ be such that $\mu(E_n)=\tau_n \mu(X_n)$. Define $E= \bigcup_{n\in \N} E_n$. 
Then, a simple calculation shows that
\[
S_{\psi,E}= S_{\sqrt{\tau}\psi,X}.
\]
Hence,
\[
||S_{\phi,E} - S_{\sqrt{\tau}\phi,X}|| 
\le
||S_{\phi,E} - S_{\psi,E}||+||S_{\sqrt{\tau}\psi,X} - S_{\sqrt{\tau}\phi,X}||
 \le 2\ve.
\]
Since $\ve>0$ is arbitrary, this shows \eqref{lya1}. 
\end{proof}

Theorem \ref{lya6} implies the following variant of Lyapunov's theorem in a spirit of Uhl's theorem \cite{Uh}, see also \cite[Theorem IX.10]{DU}. 

\begin{theorem}\label{lyu}
Let $(X,\mu)$ be a non-atomic measure space.
Suppose that $\{\phi_t\}_{t\in X}$ is a continuous Bessel family in $\mathcal H$. Let $\mathcal S$ be the set of all partial frame operators 
\begin{equation}\label{lyu1}
\mathcal S=\{ S_{\phi, E}: E \subset X \text{ is measurable} \}
\end{equation}
Then, the operator norm closure $\overline{\mathcal S} \subset B(\mathcal H)$ is convex.
\end{theorem}

\begin{proof}
Note that set
\[
\mathcal T= \{S_{\sqrt{\tau}\phi,X}: \tau \text{ is any measurable function }X \to [0,1]\}
\]
is a convex subset of $B(\mathcal H)$. Hence, its operator norm closure $\overline{\mathcal T}$ is also convex. If $\tau=\chi_E$ is a characteristic function on $E \subset X$, then $S_{\sqrt{\tau}\phi,X}=S_{\phi, E}$. Hence, $\mathcal S \subset \mathcal T$. By Theorem \ref{lya} their closures are the same $\overline{\mathcal T}=\overline{\mathcal S}$.
\end{proof}

\begin{remark}
Note that the positive operator valued measure $E \mapsto S_{\phi, E}$ does not have to be of bounded variation as required by \cite[Theorem IX.10]{DU}. Moreover, the closure of $\mathcal S$ might not be compact. Hence, Theorem \ref{lyu} can not be deduced from Uhl's theorem mentioned above.

\end{remark}

The following example shows that taking closure in Theorem \ref{lyu} is necessary.

\begin{ex}
Consider a continuous Bessel family $\{\phi_t\}_{t\in [0,1]}$ with values in $L^2([0,1])$ given by $\phi_t=\chi_{[0,t]}$. We claim that there is no measurable set $E\subset [0,1]$ such that
\begin{equation}\label{ex0}
S_{\phi,E} = \tfrac 12 S_{\phi,[0,1]}.
\end{equation}
Otherwise, we would have
\begin{equation}\label{ex1}
\frac 12 \int_0^1 | \langle f, \phi_t \rangle|^2 dt = 
\frac 12 \int_0^1 \bigg| \int_0^t f(s) ds \bigg|^2 dt =
 \int_E \bigg| \int_0^t f(s) ds \bigg|^2 dt
\qquad\text{for } f\in L^2([0,1]).
\end{equation}
For any $0\le a<b\le 1$, define $f_n(t)=n \chi_{[a,a+1/n]} - n \chi_{[b-1/n,b]}$. Then, $g_n(t)=\int_0^t f_n(s) ds$ is a piecewise linear function with knots at $(a,0)$, $(a+1/n,1)$, $(b-1/n, 1)$, and $(b,0)$, where $n> 2/(b-a)$. Applying \eqref{ex1} and taking the limit as $n\to \infty$ yields
\[
\frac{b-a}2=\frac 12 \lambda([a,b]) =\lambda(E\cap [a,b]).
\]
Since $[a,b]$ is an arbitrary subinterval of $[0,1]$, this contradicts the Lebesgue Differentiation Theorem. Hence, no set can fulfill \eqref{ex0}.
\end{ex}

We end this section by showing a more precise version of Theorem \ref{lya} for continuous Bessel families over a finite non-atomic measure space.

\begin{lemma}\label{lya6}
Suppose that $\{\phi_t\}_{t\in [0,1]}$ is a continuous Bessel family in $\mathcal H$. Let $S$ be its frame operator. Then, for any $\ve>0$ and $0< \tau < 1$, there exists a Lebesgue measurable set $E \subset [0,1]$ such that
\begin{equation}\label{lya7}
||S_{\phi,E} - \tau S|| \le \ve
\qquad\text{and}\qquad
\lambda(E) \le \tau.
\end{equation}
\end{lemma}

\begin{proof}
Let $\Sigma$ denote the set of all finite sequences of $0$'s and $1$'s. We shall construct inductively the family $\{E_\sigma\}_{\sigma\in \Sigma}$ of measurable subsets of $[0,1]$ in the following way. If $\sigma$ is an empty word, then we let $E_\sigma=[0,1]$. Assume that $E_\sigma$ is constructed for a word $\sigma$ of length $n\in\N_0$. By Theorem \ref{lya}, there exists a measurable subset $E_{\sigma 0} \subset E_\sigma$ such that
\begin{equation}\label{lya8}
||S_{\phi,E_{\sigma 0}} - \tfrac 12 S_{\phi, E_\sigma} || < 4^{-n-1}\ve.
\end{equation}
Letting $E_{\sigma 1} = E_{\sigma} \setminus E_{\sigma 0}$, we also have
\begin{equation}\label{lya9}
||S_{\phi,E_{\sigma 1}} - \tfrac 12 S_{\phi, E_\sigma} ||
=
||S_{\phi,E_{\sigma 0}} - \tfrac 12 S_{\phi, E_\sigma} || < 4^{-n-1}\ve.
\end{equation}
Moreover, by swapping these sets if necessary we also have
\begin{equation}\label{lya11}
\lambda(E_{\sigma 0}) \le \tfrac 12 \lambda(E_\sigma) \le \lambda(E_{\sigma 1}).
\end{equation}

Let $\Sigma_n$ be the set of all words in $\Sigma$ of length $n$.
For any $n\in \N$, the family $\{E_\sigma\}_{\sigma\in \Sigma_n}$ is a partition of $[0,1]$. Moreover, we have
\begin{equation}\label{lya10}
||S_{\phi,E_\sigma} - 2^{-n} S||< 2^{-n}\ve \qquad\text{for }\sigma \in\Sigma_n.
\end{equation}
To show \eqref{lya10} we will use the telescoping argument as follows. Let $\sigma_k$, $k=0,\ldots,n$, be the word consisting of the first $k$ letters of $\sigma\in \Sigma_n$. Then, by \eqref{lya8} and \eqref{lya9}
\[
\begin{aligned}
||S_{\phi,E_\sigma} - 2^{-n} S || 
&\le \sum_{k=0}^{n-1} ||2^{k+1-n} S_{\phi,E_{\sigma_{k+1}}}-  2^{k-n} S_{\phi,E_{\sigma_k}}||
\\
&= \sum_{k=0}^{n-1} 2^{k+1-n} ||S_{\phi,E_{\sigma_{k+1}}}-  \tfrac 12 S_{\phi,E_{\sigma_k}}|| < \sum_{k=0}^{n-1} 2^{k+1-n} 4^{-k-1}\ve < 2^{-n} \ve.
\end{aligned}
\]

Suppose $0<\tau<1$ has a binary expansion $\tau=\sum_{n=1}^\infty \tau(n)2^{-n}$, where $\tau(n)=0,1$. For each $n\in\N$, let 
\[
\Sigma'_n= \{ \sigma \in \Sigma_n: \sigma < \tau(1)\ldots\tau(n) \}, \qquad F_n=\bigcup_{\sigma \in \Sigma'_n} E_\sigma,
\qquad
E= \bigcup_{n=1}^\infty F_n,
\]
where $<$ denotes lexicographic order in $\Sigma_n$. By \eqref{lya10}
\[
\bigg\|S_{\phi,F_n} - \frac{\#|\Sigma'_n|}{2^n} S \bigg\| \le \sum_{\sigma\in\Sigma'_n} || S_{\phi,E_\sigma} - 2^{-n}S|| <\ve.
\]
Likewise, we use \eqref{lya11} and induction on $n$ to deduce that
\[
\lambda(F_n) \le \frac{\#|\Sigma'_n|}{2^n}.
\]
Since $F_n \subset F_{n+1}$ and $\frac{\#|\Sigma'_n|}{2^n} \to \tau$ as $n\to \infty$, we obtain \eqref{lya7}. 
\end{proof}

\begin{theorem}\label{lya13}
Let $(X,\mu)$ be a finite non-atomic measure space. Suppose that $\{\phi_t\}_{t\in X}$ is a continuous Bessel family in $\mathcal H$. Then, for any measurable function $\tau: X \to [0,1]$ and $\ve>0$, there exists a measurable subset $E \subset X$ such that
\begin{equation}\label{lya12}
||S_{\phi,E} - S_{\sqrt{\tau}\phi,X} || \le \ve
\qquad\text{and}\qquad
\mu(E) \le \int_X \tau d\mu.
\end{equation}
\end{theorem}

\begin{proof}
First we observe that Lemma \ref{lya6} generalizes to the setting of a finite non-atomic measure space $(X,\mu)$. That is, if $\{\phi_t\}_{t\in X}$ is a continuous Bessel family, then for any $\ve>0$ and $0<\tau_0<1$, there exists a measurable set $E \subset X$ such that
\begin{equation}\label{lya14}
||S_{\phi,E} - \tau_0 S|| \le \ve
\qquad\text{and}\qquad
\mu(E) \le \tau_0  \mu(X).
\end{equation}
Indeed, by Proposition \ref{p5} and Theorem \ref{p6}, there exists a continuous Bessel family $\{\psi_t\}_{t\in I }$, over the interval $I=[0,\mu(X)]$ with the Lebesgue measure $\lambda$, which has the same distribution as $\{\phi_t\}_{t\in X}$. Note that there is no need to restrict the support of $\{\phi_t\}_{t\in X}$, since $X$ is a finite measure space. Hence, by a rescaled version of Lemma \ref{lya6}, there exists a measurable subset $\tilde E \subset I$ such
\[
||S_{\psi,\tilde E} - \tau_0 S|| \le \ve
\qquad\text{and}\qquad
\lambda(\tilde E) \le \tau_0  \lambda(I).
\]
Since the correspondence between $\{\phi_t\}_{t\in X}$ and $\{\psi_t\}_{t\in I }$ is given by Carath\'eodory's Theorem \ref{p4}, there exists a measurable set $E \subset X$, which is the image of $\tilde E$ under the isomorphism of measure algebras, such that $S_{\phi,E}=S_{\psi,\tilde E}$ and $\mu(E)=\lambda(\tilde E)$. This proves \eqref{lya14}.

 Suppose that $\tilde \tau: X \to [0,1]$ is another measurable function. Then,
\[
\begin{aligned}
||S_{\sqrt{\tau}\phi,X} - S_{\sqrt{\tilde \tau}\phi,X}||
& = \sup_{||f||=1} |\langle (S_{\sqrt{\tau}\phi,X} - S_{\sqrt{\tilde \tau}\phi,X})f,f \rangle |
\\
&= \sup_{||f||=1} \bigg| \int_X (\tau(t)-\tilde \tau(t)) |\langle f, \phi_t \rangle|^2 d\mu(t)  \bigg| \le ||\tau-\tilde \tau||_\infty ||S||.
\end{aligned}
\]
Hence, it suffices to show Theorem \ref{lya13} for functions taking finitely many values.

Suppose that $\tau$ takes only finitely many values, say $s_1,\ldots,s_n$. Then, the sets $X_i=\tau^{-1}(s_i)$, $i=1,\ldots,n$, form a partition of $X$. Now we apply the above variant of Lemma \ref{lya6} for continuous Bessel family $\{\phi_t\}_{t\in X_i}$ and $0<s_i<1$, to deduce the existence of a measurable subset $E_i \subset X_i$ such that
\begin{equation}\label{lya16}
||S_{\phi,E_i} - s_i S_{\phi,X_i} || \le \ve/n
\qquad\text{and}\qquad
\mu(E_i) \le s_i  \mu(X_i).
\end{equation}
In the case of $s_i=0$ or $1$, we take $E_i=\emptyset$ or $X_i$, resp. Let $E=\bigcup_{i=1}^n E_i$. 
By the triangle inequality and \eqref{lya16},
\[
||S_{\phi,E} - S_{\sqrt{\tau}\phi,X} || \le \sum_{i=1}^n ||S_{\phi,E_i} - S_{\sqrt{\tau}\phi,X_i} ||
=  \sum_{i=1}^n ||S_{\phi,E_i} -  s_i S_{\phi,X_i} || \le \ve.
\]
Moreover,
\[
\mu(E) = \sum_{i=1}^n \mu (E_i) \le \sum_{i=1}^n s_i \mu(X_i) = \int_X \tau d\mu.
\]
This shows \eqref{lya12}.
\end{proof}

\section{Positive compact operator-valued mappings} \label{S4}

In this section we extend Theorem \ref{lyu} to the special case of POVMs given by measurable mappings with values in positive compact operators.

\begin{definition}\label{cov}
Let $\mathcal K(\mathcal H)$ be the space of positive compact operators on a separable Hilbert space $\mathcal H$. Let $(X, \mu)$ be a measure space. We say that $T = \{T_t\}_{t\in X}: X \to \mathcal K(\mathcal H)$ is {\it compact operator-valued Bessel family} if:
\begin{enumerate}
\item 
for each $f,g \in\mathcal H$, the function $X \ni t \to \lan T_t f, g \ran \in \C$ is measurable, and 
\item
there exists a constant $B>0$ such that
\[
\int_X \lan T_t f,f \ran d\mu(t) \le B ||f||^2 \qquad \text{for all }f\in \mathcal H.
\]
\end{enumerate}
\end{definition}

\begin{remark} \label{cov2}
Observe that if $\{\phi_t\}_{t\in X}$ is a continuous Bessel family, then $T_t=\phi_t \otimes \phi_t$ is an example of compact operator-valued Bessel family. This corresponds to rank 1 operator-valued mappings. Since finite rank operators are a dense subset of $\mathcal K(\mathcal H)$ with respect to the operator nom, the space $\mathcal K(\mathcal H)$ is separable. A quick extension of Proposition \ref{p1} shows that every compact operator-valued Bessel family $(T_t)_{t\in X}$ is supported on a $\sigma$-finite set. Indeed, for any $f \in \mathcal H$, $||f||=1$, by Chebyshev's inequality we have
\[
\mu( \{ t \in X:  \langle T_t f, f \ran > 1/n \} ) \le B n<\infty.
\]
The rest of argument is the same as in Proposition \ref{p1}. 

Likewise, by the Pettis Measurability Theorem, the weak measurability (i) is equivalent to strong measurability. Consequently, the mapping  $t \mapsto T_t$ is a.e. uniform limit of a sequence of countably valued measurable functions $X \to \mathcal K(\mathcal H)$. Moreover, we have the following analogue of Lemma \ref{approx}.
\end{remark}

\begin{lemma}\label{cov4}
Suppose that $\{T_t\}_{t\in X}$ is a compact operator-valued Bessel family in $\mathcal H$. For any measurable function $\tau: X \to [0,1]$, define an operator $S_{\tau T}$ on $\mathcal H$, by
\begin{equation}\label{cov5}
S_{\tau T} f = \int_X \tau(t)T_t f d\mu(t) \qquad\text{for }f\in \mathcal H.
\end{equation}
Then for every $\ve>0$, there exists a compact operator-valued Bessel family $\{R_t\}_{t\in X}$, which takes only countably many values, such that for any measurable function $\tau: X \to [0,1]$ we have
\begin{equation}\label{cov6}
||S_{\tau T} - S_{\tau R}||\le \ve.
\end{equation}
\end{lemma}

\begin{proof}
Note that $S_{\tau T}$ is a well-defined bounded positive operator with norm $\le B$.
By Remark \ref{cov2} we can assume that $(X,\mu)$ is $\sigma$-finite. Moreover, we can assume that $\mu$ is non-atomic. For any $\ve>0$, we can find a partition $\{X_{n}\}_{n\in \N}$ of $X$ such that $\mu(X_{n}) \le 1$ for all $n\in\N$. Then, we can find a countably-valued measurable function $\{R_t\}_{t\in X}$ such that 
\[
||T_t - R_t || \le \frac{\ve}{2^n} \qquad\text{for }t\in X_{n}.
\]
Using the fact that operators \eqref{cov5} are self-adjoint, we have
\[
\begin{aligned}
||S_{\tau T} - S_{\tau R}||
& = \sup_{||f||=1} |\langle (S_{\tau T} - S_{\tau R})f,f \rangle | = \sup_{||f||=1} \bigg| \int_X \tau(t) \langle (T_t-R_t)f, f \rangle  d\mu(t)  \bigg| \\
& \le \int_X ||T_t - R_t || d\mu(t) \le \sum_{n=1}^\infty \frac{\ve}{2^n} \mu(X_n) \le \ve.
\end{aligned}
\]
\end{proof}

\begin{theorem}\label{cov8}
Suppose that $\{T_t\}_{t\in X}$ is a compact operator-valued Bessel family over a non-atomic measure space $(X,\mu)$. Define a positive operator-valued measure $\Phi$ on $X$ by
\begin{equation}\label{cov8a}
\Phi(E) = \int_E T_t d\mu(t) \qquad\text{for measurable } E \subset X.
\end{equation}
Then, the closure of the range of $\Phi$ is convex.
\end{theorem}

\begin{proof}
Without loss of generality, we can assume that $X$ is $\sigma$-finite. As in the proof of Theorem \ref{lyu}, it suffices to show that for any measurable function $\tau: X \to [0,1]$ and $\ve>0$, there exists a measurable set $E \subset X$ such that
\begin{equation}\label{cov9}
||\Phi(E) - S_{\tau T}||<\ve.
\end{equation}
Let $\{R_t\}_{t\in X}$ be compact operator-valued Bessel family from Lemma \ref{cov4}. Since it takes only countably many values, there exists a partition $\{X_n\}_{n\in\N}$ of its support $\{t\in X: R_t \ne \mathbf 0\}$ such that $t\mapsto R_t$ takes constant value $R_n$ on each $X_n$. By the Bessel condition we have $\mu(X_n)<\infty$. Define values $\tau_n = \int_{X_n} \tau d\mu$, $n\in\N$. Since $\mu$ is non-atomic, we can find subsets $E_n \subset X_n$  such that $\mu(E_n)=\tau_n$. Define $E= \bigcup_{n\in \N} E_n$. 
Then, we have
\[
S_{\tau R} = \sum_{n=1}^\infty \int_{X_n} \tau(t) R_t d\mu(t) 
= \sum_{n=1}^\infty \tau_n R_n =\sum_{n=1}^\infty \mu(E_n) R_n =
 \int_E R_t d\mu(t) = S_{\chi_E R}.
\]
Applying \eqref{cov6} twice for $\tau$ and $\chi_E$ yields
\[
||\Phi(E) - S_{\tau T}|| =
||S_{\chi_E T} - S_{\tau T}|| 
\le
||S_{\chi_E T} - S_{\chi_E R }||+||S_{\tau R } - S_{\tau T}|| \le 2\ve.
\]
Since $\ve>0$ is arbitrary, \eqref{cov9} is shown. 
\end{proof}

We finish by showing that the assumption that the Bessel family $\{T_t\}_{t\in X}$ in Theorem \ref{cov8} is compact-valued is necessary.

\begin{ex}\label{cov13}
Let $I=[0,1]$ be the unit interval with the Lebesgue measure. Define Rademacher functions
\[
r_n(t) = \operatorname{sgn} \sin (2^{n+1} \pi  t), \ t\in I, n\in \N.
\]
 For any sequence $a=(a_n)_{n\in \N} \in \ell^2(\N)$, we consider a diagonal operator $\operatorname{diag}(a)$ with respect to the standard o.n. basis of $\ell^2(\N)$.
Consider operator-valued mapping $T: I \to \mathcal B(\ell^2(\N))$ given by
\[
T_t = \operatorname{diag}( r_n(t)+1)_{n\in \N}.
\]
Clearly, $\{T_t\}_{t\in I}$ satisfies properties (i) and (ii) in Definition \ref{cov}. Moreover, each $T_t$, $t\in I$, is a positive self-adjoint operator (in fact a multiple of a diagonal projection), but it is not  a compact operator. Define a POVM $\Phi$ as in \eqref{cov8a}. Since each function $r_n$ takes values $\pm 1$ on a set of measure $\frac12$, we have $\Phi(I) = \mathbf I$. We claim that $\frac12 \mathbf I$ is not in the closure of the range of $\Phi$. Indeed, suppose otherwise. Hence, there would exist a measurable set $E \subset I$ such that 
\[
\|\Phi(E)  - \tfrac12 \mathbf I \|< \tfrac14.
\]
This implies that all diagonal entries of $\Phi(E)$ lie in the interval $(1/4,3/4)$. On the other hand,
$n^{\text{th}}$ diagonal entry of $\Phi(E)$ satisfies
\[
\int_E ( r_n(t)+1) dt = \langle r_n, \chi_E \rangle +1 \to 1 \qquad\text{as }n\to \infty.
\]
This is a contradiction. Hence, the closure of the range of $\Phi$ is not convex.
\end{ex}

Example \ref{cov13} illustrates how critical it is that $\{T_t\}_{t\in X}$ is a strongly measurable function. That is, the conclusion of Theorem \ref{cov8} holds true for a general positive operator-valued Bessel family $T: X \to \mathcal B(\mathcal H)$, which is strongly measurable instead of weakly measurable and compact-valued. Such mappings $T$ can be approximated by countably valued functions. The proof follows verbatim the proofs of Lemma \ref{cov4} and Theorem \ref{cov8}.

\bibliographystyle{amsplain}

\end{document}